\documentclass[a4paper]{article}
\usepackage{a4wide}
\usepackage{amsfonts}
\usepackage{amsmath}
\usepackage{mathtools}
\usepackage{amssymb}
\usepackage{array}
\usepackage{multicol}
\usepackage{wrapfig}
\usepackage{amsthm}
\usepackage{dsfont}
\usepackage[utf8]{inputenc}
\usepackage[T1]{fontenc}
\usepackage{caption}
\usepackage{subcaption}
\usepackage{enumerate} 
\usepackage{xifthen}
\usepackage[urlcolor=blue,colorlinks=false]{hyperref}
\usepackage{algorithm}
\usepackage{graphicx}
\usepackage{color}
\usepackage{pgf}

\usepackage{multirow}
\usepackage{multicol}
\usepackage{tabularx}
\usepackage{diagbox}
\usepackage{longtable}
\usepackage{listings} 
\usepackage{cleveref}
\usepackage{tikz}
\usetikzlibrary{math}
\usetikzlibrary{patterns}
\usepackage{graphicx}
\usepackage{wrapfig}
\usepackage{changes} 
\usepackage[export]{adjustbox}

\newcommand{\N}{\ensuremath{\mathbb{N}}}

\newcommand{\Z}{\ensuremath{\mathbb{Z}}}

\newcommand{\R}{\ensuremath{\mathbb{R}}}
\newcommand{\C}{\ensuremath{\mathbb{C}}}

\newcommand{\nchoosek}[2]{{#1\choose #2}}
\newcommand{\ii}{\mathit{i}}

\newcommand{\eip}[1]{\textnormal{e}^{2\pi\ii{#1}}}

\newcommand{\Spec}{\operatorname{Spec}}

\usepackage{algpseudocode}
\usepackage{algorithm}

\newtheorem{theorem}{Theorem}[section]
\newtheorem{lemma}[theorem]{Lemma}
\newtheorem{corollary}[theorem]{Corollary}
\newtheorem{proposition}[theorem]{Proposition}
\theoremstyle{definition}
\newtheorem{remark}[theorem]{Remark}
\newtheorem{definition}[theorem]{Definition}
\newtheorem{example}[theorem]{Example}

\newenvironment{Corollary}{\goodbreak \begin{corollary}\rmfamily}{\end{corollary}}

\numberwithin{equation}{section}
\numberwithin{table}{section}
\numberwithin{figure}{section}

\newcommand{\bend}{\hspace*{0ex} \hfill \hbox{\vrule height
    1.5ex\vbox{\hrule width 1.4ex \vskip 1.4ex\hrule  width 1.4ex}\vrule
    height 1.5ex}}

\long\def\symbolfootnote[#1]#2{\begingroup%
\def\thefootnote{\fnsymbol{footnote}}\footnote[#1]{#2}\endgroup}

\crefname{lemma}{Lemma}{Lemmata}
\crefname{definition}{Definition}{Definitions}
\crefname{theorem}{Theorem}{Theorems}
\crefname{corollary}{Corollary}{Corollaries}
\crefname{equation}{}{}
\crefname{remark}{Remark}{Remarks}
\crefname{algorithm}{Algorithm}{Algorithms}
\crefname{chapter}{Chapter}{Chapters}
\crefname{section}{Section}{Sections}
\crefname{table}{Table}{Tables}
\crefname{figure}{Figure}{Figures}
\crefname{example}{Example}{Examples}
\crefname{appendix}{Appendix}{Appendices}

\renewcommand{\thefootnote}{\fnsymbol{footnote}}

\allowdisplaybreaks

\title{Two subspace methods for frequency sparse graph signals}

\author{Tarek Emmrich \and Martina Juhnke-Kubitzke \and Stefan Kunis}
\date{}

\begin{document}

\maketitle

\begin{abstract}
We study signals that are sparse in graph spectral domain and develop explicit algorithms to reconstruct the support set as well as partial components from samples on few vertices of the graph. The number of required samples is independent of the total size of the graph and takes only local properties of the graph into account. Our results rely on an operator based framework for subspace methods and become effective when the spectral eigenfunctions are zero-free or linear independent on small sets of the vertices. The latter has recently been adressed using algebraic methods by the first author.

   
\end{abstract}

\section{Introduction}

The Fourier transform has been generalised for signals defined on the vertex set of a graph \cite{ShRiVa16} and thus can serve also in the study of convolution operators in popular areas like geometric deep learning \cite{BrBrLeSzVa17,LeHuBuBrKu21}.
Special cases are the well known discrete Fourier transform (DFT), the discrete cosine transform, and the Walsh-Hadamard transform which are the Fourier transforms on the circle graph, the path graph, and the hypercube, respectively.
In all three cases fast algorithms are available for transforming any signal between spatial/vertex domain and frequency domain and these rely on divide and conquer strategies available for these specific graphs.
More recently, the a-priori assumption that only a small number of Fourier coefficients is non-zero has gained some attention and led to so-called sparse FFTs \cite{Iwen10,HaInKaPr12} which reduce the spatial sampling effort as well as the running time considerably using several quite specific properties of the DFT.
In greater generality, compressive sensing has also gained some popularity over the last two decades and its success for the DFT relies on a so-called robust uncertainty principle, see e.g., ~\cite[Thm.~12.32]{FoRa13} using tailored probabilistic arguments for \emph{random sampling}.
On the other hand, subspace methods are known to realise parameters (such as the active frequencies in a sparse sum of exponentials) as eigenvalues of certain matrices obtained from \emph{consecutive samples}, see e.g. \cite{CuLe18,StPl20}.
Their success relies on certain rank conditions which can be fulfilled by Chebotarev's uncertainty principle \cite{Tao05} for the DFT of prime length.

The primary goal of this short note is the discussion of two subspace methods for graphs which rely on 1) a recent generalisation of Chebotarev's uncertainty principle to graphs by the first author \cite{Em23} and 2) a meaningful interpretation of consecutive samples for graphs. 
Our first result is a variant of \cite{StPl20} adapted to the graph setting and uses samples in the neighbourhood of one vertex on which the eigenfunctions are assumed to not vanish. The special case of a circle graph has been considered e.g.~, in \cite{KoDr19}. 
The second method samples the graph signal in smaller neighbourhoods of several vertices, see also \cite{LiZhGaLi22} for a similar approach in so-called multi-snapshot spectral estimation.
We refer the interested reader to \cite{Tanaka_2020} for a recent survey on sampling graph signals.
However, note that in most cases the frequency-support of the signal is assumed to be known and we are only aware of \cite{KoDr19}, where the frequency-support is effectively computed for signals on the circle graph, and \cite[Proposition 2]{Marques16}, which states that samples in the neighbourhood of one vertex suffice for the identification of the support but only $\ell_0$-minimization is suggested as reconstruction algorithm.
The article is structured as follows: Section 2 shortly introduces the considered graph signal model together with two well understood examples.
All main results are contained in Section 3 and we close by a generalisation to simplicial complexes in Section 4.

\section{Preliminaries}
For a finite set of \emph{vertices} $V=[n]$ and $E\subseteq \binom{V}{2}$ we call $G=(V,E)$ a \emph{graph}.
We call two vertices $v,w$ \emph{connected}, if $\{v,w\} \in E$ and write $v \sim w$.
The \emph{adjacency matrix} $A \in \{0,1\}^{n \times n}$ of $G$ is defined by
\[
    A_{v,w} = \begin{cases} 1, \text{ if } \{v,w\}\in E, \\ 0 , \text{ otherwise}.
    \end{cases}
\]
The \emph{combinatorial Laplacian matrix} $L \in \R^{n \times n}$ is the matrix $L=D-A$, with $D_{v,v}=\operatorname{deg}(v)=|\{\{v,w\}\in E\}|$.
The Laplacian matrix is symmetric and positive semidefinite with eigenvalues $0 = \lambda_1 \leq \lambda_1 \leq \ldots \leq \lambda_n$ and orthornormal eigenvectors $n^{-1/2}(1,\hdots,1)^\top=u_1,u_2,\hdots,u_n\in\R^n$.
We write
\begin{align}
    L=U\Lambda U^*,
\end{align}
with $\Lambda=\operatorname{diag}(\lambda_1,\ldots,\lambda_n)$ and $U=(u_1,\hdots,u_n)$.
Throughout this paper, we assume the graph being \emph{simple} and \emph{connected}, which is equivalent to $A_{v,v}=0$, $v\in V$, and $\lambda_2>0$, respectively.
We define the \emph{distance} $d(v,w)$ of two vertices $v,w$ by the smallest integer $r$ such that there exist vertices $v=v_0, v_1, \ldots , v_r=w$ with $\{v_k,v_{k+1}\} \in E$.
For a vertex $v \in V$ the $k$\emph{-neighbourhood} of $v$ is the set $N(v,k)=\{w \in V : d(w,v) \leq k\}$.

We are interested in the analysis of signals defined on the vertices of the graph, i.e., vectors $f\in\R^n$ are identified with functions $f \colon V \to \R$. We are especially asking under which conditions an $s$-frequency-sparse signal
\begin{align}\label{defy}
    f=\sum_{j \in S} \beta_j u_j, \qquad \beta_j \neq 0,\; S\in\nchoosek{[n]}{s},\;s\ll n,
\end{align}
can be efficiently recovered from its samples $f(v)$, $v\in W\subset V$, $|W|\ll n$, without knowing the support $S$.
For notational convenience, let $f_W$ and $U_{W,S}$ denote the restriction of $f$ and $U$ to the rows in $W$ and columns in $S$, respectively.

\begin{example}[Circle graph]\label{ex:Circle}
 Let us start with the well-studied circle graph on $n$ vertices as
\end{example}
 \vspace{-0.2cm}
 \noindent illustrated in \cref{fig:Cn}.
 On this graph, the periodic forward shift has matrix representation
 \begin{multicols}{2}
     \[
    \tilde A_{k,\ell} = \begin{cases} 1, & \text{if } k=\ell-1 \mod n, \\ 0, & \text{otherwise},
    \end{cases}
    \] 
    \begin{wrapfigure}{l}{1\linewidth}
    \vspace{-0.7cm}
    \footnotesize
        \centering
        \begin{tikzpicture}
            \filldraw[black] (7,0) circle (2pt) node[anchor=north] {1};
            \filldraw[black] (8,0) circle (2pt) node[anchor=north] {2};
            \filldraw[black] (9,0) circle (2pt) node[anchor=north] {3};
            \filldraw[black] (10,0) circle (2pt) node[anchor=north] {4};
            \filldraw[black] (12.5,0) circle (2pt) node[anchor=north] {$n$};
            \draw (7,0) -- (8,0);
            \draw (8,0) -- (9,0);
            \draw (9,0) -- (10,0);
            \draw (10,0) -- (10.7,0);
            \draw (11.8,0) -- (12.5,0);
            \draw[bend right=20] (12.5,0) to (7.0,0);
            \node at (11.25,-0.1) {\ldots};
        \end{tikzpicture}
        \caption{The cirlce graph.}
        \label{fig:Cn}
    \end{wrapfigure}
 \end{multicols}
\noindent the adjacency matrix is $A=\tilde A+\tilde A^\top$ and Laplacian matrix is $L=2I-A$.
All three matrices are diagonalised by the Fourier matrix $U=\left(\eip{k j/n}\right)_{k,j\in[n]}$ and sparse signals with respect to this basis are well understood. We have:
\begin{enumerate}
 \item If $n$ is prime, then a classical result by Chebotarev, see e.g., \cite{Tao05}, asserts that no minor of $U$ vanishes, which implies a strong uncertainty principle and uniqueness of any $s$-sparse signal when at least $2s$ arbitrary samples are known, see also \cref{thm:Cheb} for the precise statement.
 \item If $n$ is arbitrary and $m\ge Cs\log^4n$ vertices are \emph{randomly} selected, i.e., rows of the matrix $U$ are chosen uniformly at random, then with high probability every subset $S\subset[n]$ of the columns yields a well-conditioned matrix and every $s$-sparse signal can be stably recovered by means of basis pursuit or specific greedy algorithms, see e.g., \cite[Thm.~12.32]{FoRa13}.
 \item If $n$ is arbitrary and $m\ge 2s$ \emph{consecutive} vertices are selected, then every $s$-sparse signal can be recovered by subspace methods like Prony's method, matrix pencil method, ESPRIT, or MUSIC see e.g., \cite{CuLe18,StPl20}. Stability is guaranteed if and only if the support $S$ is well-separated, i.e., if $\min_{j,\ell\in S,j\ne\ell}\min_{r\in\Z}|j-\ell+rn|>n/m$, see e.g., \cite{HoKu23}.
\end{enumerate}

\begin{example}[One-sparse signal]
We now consider an arbitrary graph and the one-sparse signal $f=\beta_j u_j$ for some $j\in[n]$.
In this case, we have
\begin{align*}
    \beta_j \lambda_j u_j(v)=(Lf)(v)=\sum_{w\sim v} \left(f(w)-f(v)\right)
\quad\text{and hence}\quad
    \lambda_j=\frac{1}{f(v)}\sum_{w\sim v} \left(f(w)-f(v)\right)
\end{align*}
if $u_j(v)\ne 0$. That is, the single active eigenvalue can be recovered from sampling the neighbouring vertices of a vertex on which the signal does not vanish. We generalise this idea to larger sparsity by means of so-called subspace methods.
Moreover, note that $u_j(v)=0$ for any nontrivial eigenvalue $\lambda_j\ne 0$ implies $$0=(\lambda_j u_j)(v)=(Lu_j)(v)=\sum_{w \sim v} u_j(w)$$
which can be read as a local symmetry property of the eigenfunction and graph.    
\end{example}

\section{Sampling and reconstruction methods}
In what follows, we discuss two subspace methods (a.k.a.~Prony's method) to first solve for the occuring eigenvalues and in a second step for the eigenfunctions in the reconstruction problem \eqref{defy}.
For the first method, we adapt the operator based approach \cite{StPl20} to the graph Laplacian resulting in sampling the $2s$-neighbourhood of one vertex for recovering an $s$-sparse signal.
The second method uses smaller neighbourhoods of several vertices and a similar idea can be found in \cite{LiZhGaLi22} under the name multi-snapshot spectral estimation.
We will analyse the number of required samples and the runtime of these approaches.
The computation of the occuring eigenvalues can be done without prior knowledge of the eigendecomposition of the Laplacian $L$. Up to a scalar factor we can also compute the eigenfunctions locally without knowledge of the whole eigendecomposition.

\subsection{Minimal sampling sets determining sparse signals}
Before considering the two specific methods, we will discuss a well-known result that $2s$ samples are necessary and, under an additional assumption, also sufficient for the reconstruction problem \eqref{defy}.
We say that a matrix $U\in\R^{n\times n}$ is \emph{Chebotarev}, if none of its minors vanishes.
Two explicit examples are the Fourier matrix in \cref{ex:Circle} if $n$ is prime, see \cite{Tao05}, and any Vandermonde-matrix with positive nodes, see \cite[Thm.~A.25]{FoRa13}.
As each minor is a polynomial in the entries of the matrix, any sufficiently random matrix is Chebotarev almost surely and we expect this property for the eigenvector matrix to hold except for specific obstructions by the graph itself.
A detailed analysis of this property by the first author can be found in \cite{Em23} for adjacency and Laplacian eigenvector matrices of random graphs.

The following result is well-known and we include a short proof for the reader's convenience.
\begin{theorem}[{e.g., \cite[Thm.~2.13]{FoRa13}}]\label{thm:Cheb}
    For any $s\in\N$, $s\le \frac{n}{2}$, $W\subset V$, $|W|\le 2s-1$, and $S_f\in \nchoosek{[n]}{s}$, there exist a set $S_g\in \nchoosek{[n]}{s}$, and coefficients $\hat f,\hat g\in\R^n$ with $\operatorname{supp}\hat f\subset S_f$, $\operatorname{supp}\hat g\subset S_g$,
    \[
    \sum_{k\in \tilde S_f} \hat f_k u_k=f \ne g = \sum_{k\in \tilde S_g} \hat g_k u_k,\quad \text{but} \quad f_W=g_W.
    \]

    On the other hand, if the eigenvector matrix $U \in \C^{n \times n}$ is Chebotarev, then for any $W\subset V$, $|W|\ge 2s$, the samples at $W$ uniquely determine any $s$-sparse signal, i.e.,
    \[
    f_W=g_W \quad \text{implies} \quad f = g.
    \]
    \end{theorem}
\begin{proof}
 Without loss of generality, let $S_f\cap S_g=\emptyset$ and consider the restricted eigenvector matrix
 \[
  U_{W,S_f\cup S_g}\in \R^{|W|\times 2s}
 \]
 which has a nontrivial kernel in the first case and has full column rank by assumption in the second case.
\end{proof}

An algorithm using such a minimal sampling set $W$ would compute the eigenvector matrix $U_W\in\R^{|W|\times n}$ first and then use a combinatorial search for the smallest set $S\subset[n]$ which allows to interpolate all $|W|$ samples. For each candidate set $S$ out of the $\binom{n}{s}$ possibilities, this asks for solving the overdetermined linear system of equations $(u_k(v))_{v\in W,k\in S} \cdot (\hat{f}_k)_{k\in S} = (f(v))_{v\in W}$.

\subsection{Sampling one neighbourhood}
The first method builds on an operator-based formulation of Prony's method as considered in \cite{StPl20}. We rely on the two facts that the powers of the Laplace operator applied to the signal and evaluated at a fixed vertex
\begin{enumerate}
  \item can be computed recursively from neighbouring samples via
  \begin{align*}
    (L^kf)(v)=\sum_{w \sim v} (L^{k-1}f)(v)-(L^{k-1}f)(w),\qquad L^0f=f,
\end{align*}
 \item and that they constitute an exponential sum (with respect to the eigenvalues)
 \begin{align*}
    (L^kf)(v)=\sum_{j \in S} \alpha_j \lambda_j^k,\qquad \alpha_j\coloneqq \beta_j u_j(v).
\end{align*}
\end{enumerate}

\begin{theorem}\label{maintheoremintroduction}
   Let $G$ be a graph, $U \in \R^{n \times n}$ be the matrix of eigenvectors of $L$, and $v\in V$ be some fixed vertex. Then every $s$-sparse signal $f=\sum_{j \in S}\beta_j u_j$ with $u_j(v)\ne 0$, $j\in S$, can be recovered from the samples
   $f(w)$, $w\in W:=N(v,2s-1)$. In addition to the support set $S$ also the restricted eigenfunctions $\beta_ju_j(w)$, $w\in N(v,s)$, can be reconstructed.
\end{theorem}
\begin{proof}
    The samples $f(w)$, $w\in W$, allow to calculate $g(k):=(L^k f)(v)$ for $k=0,\hdots,2s-1$ and the remaining task is to solve the equations
    \begin{align*}
    \sum_{j \in S} \alpha_j \lambda_j^k = g(k)
    \end{align*}
    for the unknowns $\lambda_j$, e.g., via Prony's method \cite{StPl20}, or one of its variants.
    
    Beyond this first step, the recovery of the weighted eigenfunctions relies on
    \begin{align}\label{eq:Proj}
     \left(\prod_{k \in S \setminus \{j\}} \frac{L-\lambda_k I}{\lambda_j-\lambda_k}\right)f
     &=\sum_{i \in S}\beta_i \prod_{k \in S \setminus \{j\}} \frac{\left(L-\lambda_k I\right)u_i}{\lambda_j-\lambda_k} 
     =\sum_{i \in S}\beta_i \prod_{k \in S \setminus \{j\}} \frac{\lambda_i-\lambda_k}{\lambda_j-\lambda_k}u_i 
     =\beta_j u_j 
    \end{align}
    which can be evaluated for all vertices $w\in V$ with $d(w,v)\le s$ since the product of the shifted Laplace operators on the leftmost expression uses the $(s-1)$-neighbourhood of this vertex $w$ which stays within the $(2s-1)$-neighbourhood of the fixed vertex $v$.
\end{proof}

\begin{algorithm}
\caption{Recovery from samples in one neighbourhood}
\begin{algorithmic}\label{alg:one}
\State \textbf{Input:} Fixed vertex $v\in V$ with $f(v)\ne 0$, sparsity $s\in\N$, samples $f(w)$, $w\in N(v,2s-1)$
\State Compute $g(k)=(L^k f)(v)$, $k=0,\hdots,2s-1$
\State Solve the Hankel linear system
\[
  Hp=0,\quad \text{where}\; p_s=1,\; H\in\R^{s\times (s+1)},\; H_{k,\ell}=g(k+\ell)
\]
\State Compute the eigenvalues $\lambda_1,\hdots,\lambda_s$ of the companion matrix
\[
P=
 \begin{pmatrix}
  0       & \hdots & \hdots & 0      & -p_0  \\
  1       & \ddots &        & \vdots & -p_1  \\
  0       & \ddots & \ddots & \vdots & \vdots\\
  \vdots  & \ddots & \ddots & 0      & -p_{s-2}      \\
  0       & \hdots & 0      & 1 & -p_{s-1}         
 \end{pmatrix}
\]
\State Compute the local eigenvectors $\beta_j u_j(w)$, $j=1,\hdots,s$, $w\in N(v,s)$, via \eqref{eq:Proj}.
\State \textbf{Output:} Eigenvalues and local eigenvectors.
\end{algorithmic}
\end{algorithm}
\begin{remark}[Sampling effort and computational complexity]
Let $m$ denote the total number of edges in the subgraph induced by the vertices in the neighbourhood $N(v,2s-1)$.
Then the Laplace matrix restricted to these vertices has $O(m)$ nonzero entries and thus the first step needs at most $O(sm)$ operations. We note in passing that the factor $s$ can be removed if the neighbourhood of $v$ grows exponentially fast.
The  second and third step take $O(s^3)$ floating point operations.
Similar to the first step, the last step in \eqref{eq:Proj} takes at most $O(s^2 m)$ floating point operations.
\end{remark}

\begin{example}\label{ex:Pn}
We will illustrate \cref{alg:one} for the path graph as illustrated in \cref{fig:Pn}.
\begin{figure}[ht]
\footnotesize
        \centering
        \begin{tikzpicture}
            \filldraw[black] (0,0) circle (2pt) node[anchor=north] {1};
            \filldraw[black] (1,0) circle (2pt) node[anchor=north] {2};
            \filldraw[black] (2,0) circle (2pt) node[anchor=north] {3};
            \filldraw[black] (3,0) circle (2pt) node[anchor=north] {4};
            \filldraw[black] (5.5,0) circle (2pt) node[anchor=north] {$n+1$};
            \draw (0,0) -- (1,0);
            \draw (1,0) -- (2,0);
            \draw (2,0) -- (3,0);
            \draw (3,0) -- (3.7,0);
            \draw (4.8,0) -- (5.5,0);
            \node at (4.25,-0.1) {\ldots};
        \end{tikzpicture}
        \caption{The path graph.}
        \label{fig:Pn}
    \end{figure}
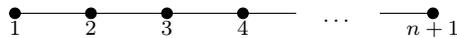

    The path graph on $n$ vertices has edge set $E=\{\{v,v+1\}:v=1,\hdots,n\}$. The Laplace matrix is given by
    \begin{align*}
        L=\begin{pmatrix}
             1     & -1     & 0      & \hdots & 0\\
            -1     & 2      & \ddots & \ddots & \vdots\\
            0      & \ddots & \ddots & \ddots & 0 \\
            \vdots & \ddots & \ddots & 2      & -1\\
            0      & \hdots & 0      & -1     & 1
        \end{pmatrix}.
    \end{align*}
    Its simple eigenvalues and eigenvectors are
    \begin{align*}
        \lambda_j=2-2\cos\Big(\frac{\pi (j-1)}{n}\Big)\in[0,4),\;
        u_{j}(v)=\frac{\sqrt{2-\delta_{1,j}}}{\sqrt{n}} \cdot \cos \left(\frac{\pi (j-1)(2v-1)}{2n}\right),\; v,j=1, \hdots, n.
    \end{align*}

    \begin{figure}[ht]
        \centering
        \includegraphics[width=4.5cm]{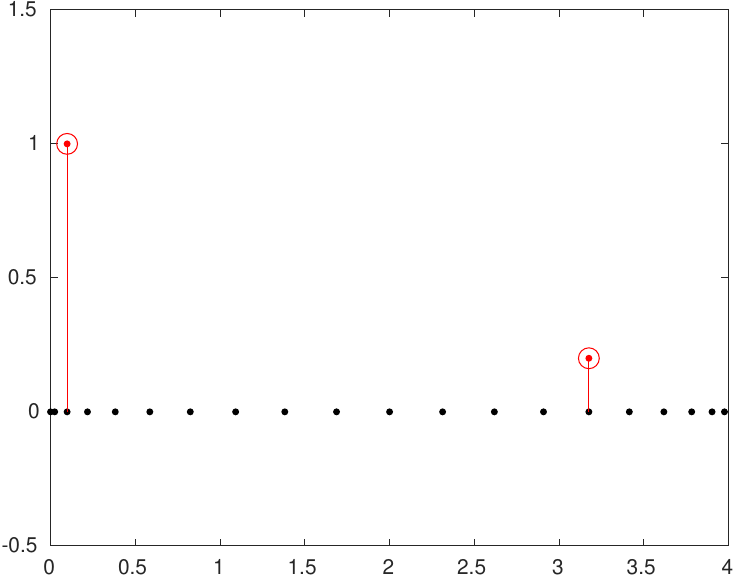}
        \includegraphics[width=4.5cm]{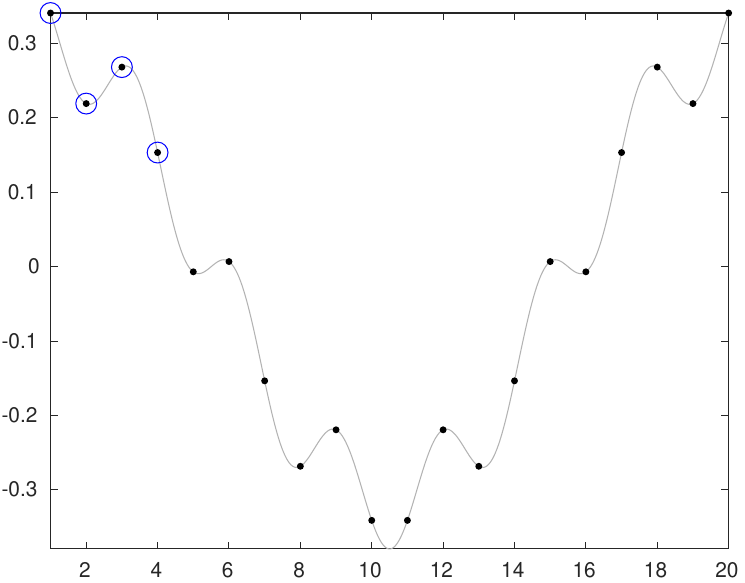}
        \includegraphics[width=4.5cm]{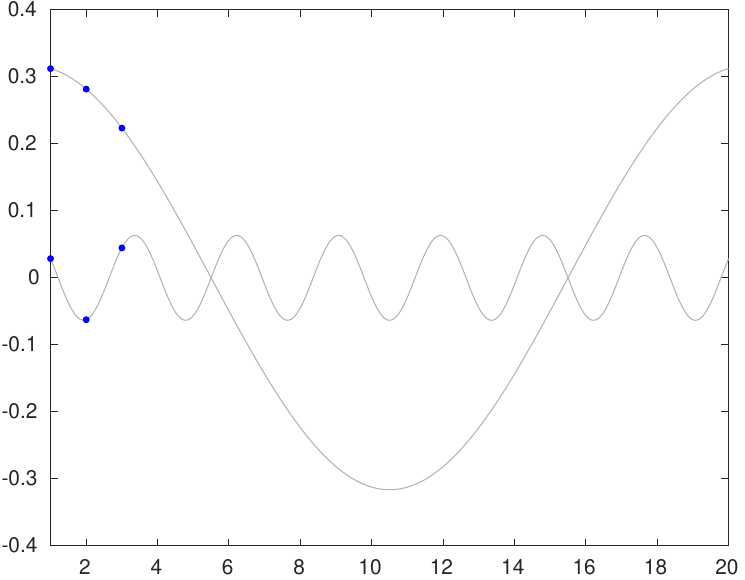}
        \caption{From left to right: Sparse coefficient vector on active eigenvalues, graph signal on the vertices $1,\hdots,20$ together with its used samples on the vertices $W=\{1,2,3,4\}$ (blue circles), and the decomposition via \eqref{eq:Proj} into $1u_3$ and $\frac15 u_{15}$ on the vertices $W=\{1,2,3\}$ (blue dots). All continuous curves are only meant to support visual inspection.}
        \label{fig:SigPn}
    \end{figure}
    To give an explicit example, let $n=20$, $s=2$, $S=\{3,15\}$, $f=1 u_3 + \frac15 u_{15}$, and $v=1$. \cref{fig:SigPn} shows the sparse coefficient vector (on the active eigenvalues $\lambda_3$ and $\lambda_{15}$) and the graph signal $f$ on the vertices $1,\hdots,20$, respectively.
    \cref{alg:one} uses the samples $(f(1),\hdots,f(4))\approx(0.34,0.22,0.27,0.15)$, see \cref{fig:SigPn} (right, blue circles), computes in its first step the values $g(0),\hdots,g(3)\approx(0.34,0.12,0.29,0.92)$, sets up the Hankel matrix and from its kernel vector the companion matrix, i.e.,
    \begin{align*}
      H\approx\begin{pmatrix}
          0.34&0.12&0.29\\
          0.12&0.29&0.92
      \end{pmatrix},\qquad
      P\approx \begin{pmatrix}
          0.00 & -0.31\\
          1.00 & 3.27
      \end{pmatrix}.
    \end{align*}
    Here, the absolute error between the numerically computed eigenvalues of the companion matrix and the true active eigenvalues $\lambda_3$, $\lambda_{15}$ as well as the maximal error between the numerically computed samples via \eqref{eq:Proj} and the true components $1u_3$ and $\frac15 u_{15}$ is smaller than $10^{-14}$.
\end{example}

\begin{remark}
 The condition $u_j(v)\ne 0$ is much weaker than being Chebotarev, since only the size-one minors are concerned. Moreover, note that this condition is true for all eigenfunctions and all vertices if the characteristic polynomial of the Laplace matrix is irreducible over the rationals up to the trivial factor for the eigenvalue zero, see \cite[Thm.~6]{Em23} for details and a discussion for random graphs.
\end{remark}

\begin{remark}[Multiple eigenvalues]
The leftmost bracketed term in \eqref{eq:Proj} is the projection operator to the eigenspace of the semisimple eigenvalue $\lambda_j$. If this subspace has dimension larger than one, we only recover this projection and no individual decomposition in a basis of this subspace.
An example is given by the circle graph in \cref{ex:Circle}, where the $j$-th and the $(n+2-j)$-th column of the Fourier matrix belong to the same real eigenvalue of the Laplacian matrix.
This can be seen as the discrete analogon to the two eigenfunctions $t\mapsto\exp(\pm \lambda t)$ for the Laplace operator on the real line with eigenvalue $\lambda$. Even larger multiplicities occur e.g., for the Laplace-Beltrami operator on the sphere and seem to be likely for specific discretisations of the sphere and small eigenvalues.

\end{remark}

\subsection{Sampling several neighbourhoods}
Since, depending on the graph, the size of the set $W=N(v,2s-1)$ can grow relatively fast with $s$, we want to improve our above approach in the way that we use less samples.
The decrease in the size of the neighbourhood will be compensated by considering several neighbourhoods, see also \cite{LiZhGaLi22} for a similar idea in so-called multi-snapshot spectral estimation.

\begin{theorem}\label{generalcase}
  Let $G$ be a graph, $U \in \R^{n \times n}$ be the matrix of eigenvectors of $L$, fix some vertices $v_1,\hdots,v_t\in V$ and radii $r_1,\hdots,r_t\in \N$, $r=r_1+\hdots+r_t$.
  Then every $s$-sparse signal $f=\sum_{j \in S}\beta_j u_j$ can be recovered from the samples
  $f(w)$, $w\in W:=\cup_{i=1}^t N(v_i,s-1+r_i)$ if the matrix
  \begin{align*}
    B=\begin{pmatrix}
        \alpha_{1,j_1} & \alpha_{1,j_2} & \ldots & \alpha_{1,j_s} \\
        \lambda_{j_1}\alpha_{1,j_1} & \lambda_{j_2}\alpha_{1,j_2} & \ldots & \lambda_{j_s}\alpha_{1,j_s} \\
        \vdots &\vdots & & \vdots \\
        \lambda_{j_1}^{r_1-1}\alpha_{1,j_1} &\lambda_{j_2}^{r_1-1}\alpha_{1,j_2} & \ldots & \lambda_{j_s}^{r_1-1}\alpha_{1,j_s} \\
        \alpha_{2,j_1} & \alpha_{2,j_2}& \ldots & \alpha_{2,j_s} \\
        \lambda_{j_1}\alpha_{2,j_1} &\lambda_{j_2}\alpha_{2,j_2} & \ldots &\lambda_{j_s}\alpha_{2,j_s}\\
        \vdots & \vdots & & \vdots \\
        \lambda_{j_1}^{r_t-1}\alpha_{t,j_1} & \lambda_{j_2}^{r_t-1}\alpha_{t,j_2} & \ldots & \lambda_{j_s}^{r_t-1}\alpha_{t,j_s}
    \end{pmatrix}\in\R^{r\times s},\qquad \alpha_{i,j}=\beta_j u_j(v_i),
\end{align*}
has full column rank $s$.
\end{theorem}
Our result in \cref{maintheoremintroduction} is the case $t=1$, $r_1=s$, where $B$ has full rank by the factorization $B=(\lambda_j^k)_{k=0,\hdots,r_1-1;s\in S} \cdot \mathrm{diag}((\alpha_{1,j})_{j\in S})$ in the product of a Vandermonde matrix with respect to the eigenvalues and a diagonal matrix with nonzero diagonal entries if $u_j(v_1)\ne 0$ and $\beta_j\ne 0$.
\begin{proof}
    The main idea is to use the equations   
\begin{align*}
    g_i(k):=(L^kf)(i)=\sum_{j \in S} \alpha_{i,j}\lambda_j^k
\end{align*}
for not only one vertex $v$ but all vertices $v_1,\hdots,v_t$. We set up a stacked Hankel matrix of these samples which allows for the following Vandermonde factorization
\begin{align*}
  \tilde H=
    \begin{pmatrix}
        g_1(0) & g_1(1) & \ldots & g_1(s-1) \\
        g_1(1) & g_1(2) & \ldots & g_1(s) \\
        \vdots & \vdots & & \vdots \\
        g_1(r_1-1) & g_1(r_1) & \ldots & g_1(s-2+r_1) \\
        g_2(0) & g_2(1) & \ldots & g_2(s-1) \\
        g_2(1) & g_2(2) & \ldots & g_2(s) \\
        \vdots & \vdots & & \vdots \\
        g_t(r_t-1) & g_t(r_t) & \ldots & g_t(s-2+r_t)
        \end{pmatrix}
        =B \cdot C,\qquad C=\begin{pmatrix}
            \lambda_j^k
        \end{pmatrix}_{j\in S,k=0,\hdots,s}.
\end{align*}
First note that the kernel of the Vandermonde matrix $C$ is one dimensional and $Cp=0$ can be rephrased as the polynomial $p=\sum_{k=0}^s p_k \lambda^k$
having exactly the roots $\lambda_j$, $j\in S$.
If the matrix $B$ has full column rank $s$, then $\ker\tilde H=\ker C$.
\end{proof}

The second extremal case is $r_1=\hdots=r_t=1$ and thus $r=t$ and for simplicity also $t=s$, where the matrix $B$ has full column rank if and only if the eigenvectors restricted to $W$ are linearly independent. A sufficient condition for this to happen is that the full eigenvector matrix $U$ is Chebotarev.
\begin{Corollary}\label{maintheoremintroduction2}
Let $G$ be a graph, $U \in \R^{n \times n}$ be the matrix of eigenvectors of $L$, and fix some vertices $v_1,\hdots,v_s\in V$.
  Then every $s$-sparse signal $f=\sum_{j \in S}\beta_j u_j$ can be recovered from the samples
  $f(w)$, $w\in W:=\cup_{i=1}^s N(v_i,s)$ if the matrix
  \[U_{W,S}=(u_j(v_i))_{i=1,\hdots,s;j\in S}\in\R^{s\times s}\] is regular.
\end{Corollary}

\begin{remark}[Rank of $B$]
If $r_1=r_2=\hdots=r_t$, then the matrix $B$ in \cref{generalcase} can be written as column-wise Kronecker product
\begin{align*}
    B&=\left(B_{j_1}\otimes\Lambda_{j_1}\hdots B_{j_s}\otimes\Lambda_{j_s}\right),\qquad
    B_{j}=\begin{pmatrix}
        \alpha_{1,j}\\
        \vdots\\
        \alpha_{t,j}
    \end{pmatrix},\quad
    \Lambda_{j}=\begin{pmatrix}
        \lambda_j^0\\
        \vdots\\
        \lambda_j^{r_1-1}
    \end{pmatrix}.
\end{align*}
Together with \cite[Eq.(2.2)]{KoBa09}, $\tilde B=\left(B_{j_1},\hdots,B_{j_s}\right)$, and $\tilde \Lambda=\left(\Lambda_{j_1},\hdots,\Lambda_{j_s}\right)$, we have the identity
$B^\top B=\tilde B^\top \tilde B \circ \tilde \Lambda^\top \tilde \Lambda$, where $\circ$ denotes the Hadamard/entrywise product. The rank of $B^\top B$ (and thus also of $B$) is bounded from above by the product of the ranks of the factors $\tilde B$ and $\tilde \Lambda$ -- and we suspect that this is attained generically, see also \cite{DaDi23}. Other recent results \cite{HoYa20} show that $B$ has full rank as soon as certain rank and Kruskal rank conditions on the factors are met -- which however would become effective in our situation only if $\min\{r_1,t\}\ge s$ instead of $t\cdot r_1\ge s$.
\end{remark}

\begin{remark}[Sampling effort]
\cref{alg:one} can be adapted by replacing the Hankel matrix by the stacked Hankel matrix and skipping the last step. In case of \cref{maintheoremintroduction2}, we expect that much less samples of the graph signal are taken -- $s$ times an $s$-neighbourhood in contrast to one $(2s-1)$-neighbourhood.
\end{remark}

\begin{example}
For the path graph one might be tempted to use the method of Theorem \ref{generalcase} in an efficient way that needs less samples. For example any 6-sparse signal $f$ and the sequence $r_1=3$, $r_2=2$, $r_3=1$ leads to the sampling set $N(1,8)\cup N(2,7) \cup N(3,6)= \{1,\ldots,9\}$, which however is too small for recovery by \cref{thm:Cheb}. 
\end{example}
\vspace{-0.25cm}
Moreover, there are cases where the neighbourhood contains enough vertices but the matrix $B$ is still singular. Let $n\ge 6$ and $G$ be the graph as shown in \cref{UmbrellaGraph}.
Setting $s=r=3$, $r_n=2$, and $r_{n-1}=1$ samples each graph signal at all vertices but the matrix
\begin{multicols}{2}
\[
B=
\begin{pmatrix}
    \alpha_{n,j_1} & \alpha_{n,j_2} & \alpha_{n,j_3} \\
    \lambda_{j_1} \alpha_{n,j_1} & \lambda_{j_2} \alpha_{n,j_2} & \lambda_{j_3} \alpha_{n,j_3} \\
    \alpha_{{n-1},j_1} & \alpha_{{n-1},j_2} & \alpha_{{n-1},j_3}
\end{pmatrix}
\]
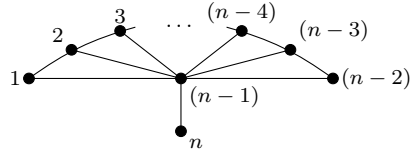
\begin{wrapfigure}{l}{1\linewidth}
\footnotesize
    \vspace{-0.7cm}
        \centering
        \begin{tikzpicture}
            \filldraw[black] (3,0) circle (2pt) node[anchor=north west] {$(n-1)$};
            \filldraw[black] (3,-0.7) circle (2pt) node[anchor=north west] {$n$};
            \filldraw[black] (1,0) circle (2pt) node[anchor=east] {1};
            \filldraw[black] (5,0) circle (2pt) node[anchor=west] {$(n-2)$};
            \filldraw[black] (1.56,0.38) circle (2pt) node[anchor=south east] {2};
            \filldraw[black] (2.2,0.63) circle (2pt) node[anchor=south] {3};
            \filldraw[black] (4.44,0.38) circle (2pt) node[anchor=south west] {$(n-3)$};
            \filldraw[black] (3.8,0.63) circle (2pt) node[anchor=south] {$(n-4)$};
            \draw[bend right=10] (2.4,0.68) to (1,0);
            \draw[bend right=10] (5,0) to (3.6,0.68);
            \draw (1,0) -- (3,0);
            \draw (3,0) -- (5,0);
            \draw (3,0) -- (3,-0.7);
            \node at (3,0.7) {\ldots};
            \draw (1.56,0.38) -- (3,0);
            \draw (2.2,0.63) -- (3,0);
            \draw (4.44,0.38) -- (3,0);
            \draw (3.8,0.63) -- (3,0);
        \end{tikzpicture}
        \caption{The umbrella graph.}
        \label{UmbrellaGraph}
    \end{wrapfigure}
\end{multicols}
\noindent is singular since $\lambda_{j}\alpha_{n,j}=\alpha_{n,j}-\alpha_{{n-1},j}$.

\section{Generalization to simplicial complexes}

In this section we will generalize the results for sparse signals on graphs to simplicial complexes. We will see that the special structure of the Laplacian and its eigenspaces allows to even improve the results for higher-dimensional simplicial complexes.

Similar to graphs there are Laplacian matrices or Laplacian operators for simplicial complexes \cite{MuHoJo22}. Let $\Delta$ be a simplicial complex on $[n]$, i.e., $\Delta$ is a collection of subsets of $[n]$ that is closed under inclusion. We denote by $\partial$ the boundary operator of the chain complex $C_\bullet(\Delta)$ with $C_k(\Delta)$ being the free $\mathbb{R}$-vectorspace of $k$-chains whose basis is given by the $k$-faces of $\Delta$, i.e., by elements in $\Delta_k=\binom{[n]}{k+1}\cap \Delta$. We use $\partial_k$ to denote the $k$-th operator of $\partial$, i.e.,
\[
\partial_k(\langle v_0, \ldots, v_k\rangle) = \sum_{i=0}^k {(-1)}^i \cdot \langle v_0,\ldots,\bar{v_i},\ldots,v_k\rangle,
\]
where $\bar{v_i}$ means that the element $v_i$ is omitted. The $k$\emph{-th Laplacian operator} $L_k(\Delta)$, or just $L_k$, if the complex is clear from the context, is defined by
\[
L_k=\underbrace{\partial_{k}^* \circ \partial_k}_{L_k^\textnormal{DN}} + \underbrace{\partial_{k+1} \circ \partial_{k+1}^*}_{L_k^\textnormal{UP}},
\]
where we also use $L_k$, $L_k^\textnormal{UP}$ and $L_k^\textnormal{DN}$ for the matrix representation of the $k$-th Laplacian (up or down)-operator. Note that $L_0=L_0^\textnormal{UP}$ is the usual graph Laplacian of the 1-skeleton of $\Delta$ and that $\partial_1$ is the vertex-edge incidence matrix. It is well-known that
\[\ker(L_k) \cong H_k(\Delta,\R),\]
where $H_k(\Delta)=\operatorname{ker}(\partial_k)/\operatorname{im}(\partial_{k+1})$ denotes the $k$-th homology group of $\Delta$ with coefficients in $\mathbb{R}$. We are interested in the eigenvalues and eigenvectors of $L_k$. The former are again real because of the symmetry and they are nonnegative. Furthermore, we know that
\[\Spec_{\neq 0}(L_k)=\Spec_{\neq 0}(L_k^\textnormal{UP}) \cup \Spec_{\neq 0}(L_k^\textnormal{DN})\]
since the operators are self-adjoint and mutually annihilating, i.e., $\partial_{k} \circ \partial_{k+1}=0$. Thus we can decompose 
\[C_k(\Delta)=\operatorname{Eig}_{\neq 0}(L_k^\textnormal{UP}) \oplus \operatorname{Eig}_{\neq 0}(L_k^\textnormal{DN}) \oplus H_k(\Delta,\R).\]
We write again $L_K=U\Lambda U^*$ with columns $u_i$ of $U$. Given a simplicial complex $\Delta$, we say that $k$-faces $\sigma,\tau\in \Delta$ have distance $d$, denoted by $d(\sigma,\tau)=d$, if $d\in \N$ is minimal such that there exists a sequence $\sigma=\tau_0,\ldots,\tau_{d}=\tau$ of $k$-faces of $\Delta$ with $|(\tau_i \cap \tau_{i+1})|=k$ for all $0\leq i\leq d-1$. We define the $d$-neighbourhood of $\sigma$ as
\[N(\sigma,d)\coloneqq\{\tau \in \binom{[n]}{k+1}~:~ d(\sigma,\tau)\leq d\}.\]
Basically, the same theorems as in the graph case do hold. We want to analyze how well we can recover an $s$-sparse sum of eigenvectors
\[f=\sum_{i \in S} \beta_i u_i\]
of eigenvectors by sampling $f$ at relatively few $k$-faces. As in the graph case, we will need to sample at a neighbourhood of a $k$-face, but we can improve on the size of the neighbourhood. We start with the reconstruction of $L_k^\textnormal{UP}$- or $L_k^\textnormal{DN}$-signals analog to Theorem \ref{maintheoremintroduction}.
\begin{theorem}
    Let $\Delta$ be a simplicial complex, $L_k$ its $k$-th Laplacian matrix, $L_k=U\Lambda U^*$ and $T \in \{L_k^\textnormal{UP},L_k^\textnormal{DN}\}$. Every $s$-sparse signal
    \[f=\sum_{j \in S}\beta_j u_j, \text{ for } u_j \in \operatorname{Eig}_{\neq 0} (T) \text{ for all }j \in S, \]
    with only one eigenvector per eigenvalue can be recovered at the face $\sigma \in \Delta_k$ by sampling the values $f(\tau)$ for all $\tau\in N(\sigma, 2s-1)$, if $u_j(\sigma)\neq 0$ for all $j \in S$.
\end{theorem}
We decompose the set $S$ into $S=S^\textnormal{UP} \cup S^\textnormal{DN} \cup S^0$, according to $i \in S^\textnormal{UP}$ iff $u_i \in \operatorname{Eig}(L_k^\textnormal{UP})$, $i \in S^\textnormal{DN}$ iff $u_i \in \operatorname{Eig}(L_k^\textnormal{DN})$ and $i \in S^0$ iff $u_i \in H_k(\Delta,\R)$.
\begin{theorem}
    Let $\Delta$ be a simplicial complex, $L_k$ its $k$-th Laplacian matrix, $L_k=U \Lambda U^*$. Let
    \[
    f=\sum_{i \in S} \beta_i u_i
    \]
    be a signal with $u_i \notin \operatorname{ker}(L_k)$. We can recover the representation of $f$ by recovering the representations of $L_k^\textnormal{UP}(f)$ and $L^\textnormal{DN}_k(f)$. This requires sampling in the $2\cdot\operatorname{max}\{|S^\textnormal{DN}|,|S^\textnormal{UP}|\}$-neighbourhood.
\end{theorem}
\begin{proof}
By applying $L_k^\textnormal{UP}$ and $L_k^\textnormal{DN}$ to $f$ we obtain two sums with a reduced number of summands, namely
\[f^\textnormal{UP}=\sum_{i \in S^\textnormal{UP}} \beta_i \lambda_i u_i \text{ and } f^\textnormal{DN}=\sum_{i \in S^\textnormal{DN}} \beta_i \lambda_i u_i.\]
This gives us two sums that can be recovered individually. A priori we do not know the sizes of the sets $S^\textnormal{UP}$ and $S^\textnormal{DN}$. We know for sure that $|S^\textnormal{UP}| + |S^\textnormal{DN}| \leq s$ and thus we know that one of the set sizes is bounded by $\frac{s}{2}$. In a first step, we therefore assume that they are $\frac{s}{2}$-sparse. One of these sums will be recovered and will also reveal the real number of summands in the sum.  We hence know an upper bound for the number of summands in the other sum and can recover this sum in a second step. By applying $L^\textnormal{UP}_k$ and $L^\textnormal{DN}_k$ we lose the contribution of the eigenvectors corresponding to the eigenvalue 0. The sampling of $f^\textnormal{UP}$ and $f^\textnormal{DN}$ in an $r$-neighbourhood requires the knowledge of $f$ in an $(r+1)$-neighbourdhood. Hence we need the $2r$- and not the $(2r-1)$-neighbourdhood from the prior theorem.
\end{proof}

We continue with the most general statement. The matrix $B$ will is the same matrix as in Theorem \ref{generalcase}.
\begin{theorem}
    Let $\Delta$ be a simplicial complex, $L_k$ its $k$-th Laplacian matrix, $L_k=U\Lambda U^*$ and $T \in \{L_k^\textnormal{UP},L_k^\textnormal{DN}\}$. Every $s$-sparse signal
    \[f=\sum_{j \in S}\beta_j u_j, \text{ for } u_j \in \operatorname{Eig}_{\neq 0} (T) \text{ for all }j \in S, \]
    with only one eigenvector per eigenvalue can be recovered by sampling the values $f(\tau)$ for $\tau \in W$ if there exists a set $R\subseteq \Delta_k$ and integers $r_\sigma \geq 1$ for $\sigma \in R$, with $\sum_{\sigma \in R}r_\sigma=s$ and $N(\sigma, s-1+r_\sigma)\subseteq W$ and if the corresponding matrix $B$ is regular.
\end{theorem}

Reddy and Chepuri \cite{Reddy} claim that any sparse signal can be uniquely recovered by sampling in some neighbourhood of a simplex. However, they do not give explicit algorithms and do not use the smaller sampling trick using $f^\textnormal{UP}$ and $f^\textnormal{DN}$. Also the full recovery of the kernel in any case seems to be daring.

\bibliographystyle{abbrv}
\bibliography{refs}

\end{document}